\documentclass[11pt]{amsart}

\usepackage{latexsym, amssymb, amsfonts, amsthm, amsmath, xspace, Macros, pdfsync}

\title{Coprime invariable generation and minimal-exponent groups}
\author{Eloisa Detomi, Andrea Lucchini and Colva M. Roney-Dougal}
\address{
(1) Eloisa Detomi and Andrea Lucchini, Universit\`a degli Studi di Padova,  Dipartimento di Matematica, Via Trieste 63, 35121 Padova, Italy\newline
(2) Colva M. Roney-Dougal, University of St Andrews, Mathematical Institute, St Andrews, Fife KY16 9SS, Scotland}

\begin{document}

\begin{abstract}
A finite group $G$ is \emph{coprimely-invariably generated} if there exists a set of generators $\{g_1, \ldots , g_u\}$ of $G$ with the property that the orders $|g_1|, \ldots , |g_u|$ are pairwise coprime and that for all $x_1, \ldots, x_u \in G$ the set $\{g_1^{x_1}, \ldots, g_u^{x_u}\}$ generates $G$.

We show that if $G$ is coprimely-invariably generated, then $G$ can be generated with three elements, or two if $G$ is soluble, and that $G$ has zero presentation rank. As a corollary, we show that if $G$ is any finite group such that no proper subgroup has the same exponent as $G$, then $G$ has zero presentation rank. Furthermore, we show that every finite simple group is coprimely-invariably generated.

Along the way, we show that for each finite simple group $S$, and for each partition $\pi_1, \ldots, \pi_u$ of the primes dividing $|S|$, the product of the number $k_{\pi_i}(S)$  of conjugacy classes of $\pi_i$-elements satisfies
$$\prod_{i=1}^u k_{\pi_i}(S) \leq \frac{|S|}{2|\out S|}.$$
\end{abstract}

\maketitle

\section{Introduction}

Following \cite{dix2}
and \cite{ig}, we say that
a subset $\{g_1, \ldots , g_u\}$ of a finite group $G$  \emph{invariably generates} $G$ if
 $\{g_1^{x_1}, \ldots , g_u^{x_u}\}$ generates $G$ for every choice of $x_i \in G$.

\begin{definition}
A finite group $G$ is \emph{coprimely invariably generated} if there exists a set of invariable generators $\{g_1, \ldots , g_u\}$ of $G$ with the property that the orders $|g_1|, \ldots , |g_u|$ are pairwise coprime.
\end{definition}


Our main result says that a coprimely invariably generated group can be generated
with very few elements. Let $d(G)$ denote the minimal number of generators of
$G$.

\begin{thm}\label{3}
Let $G$ be a coprimely invariably generated group. Then $d(G) \leq 3$.
\end{thm}

Notice that coprime invariable generation  is the combination of two properties: the
existence of an invariable generating set and the existence of
a set of generators of coprime orders. It is worth noticing than
neither of these properties suffices to obtain an
upper bound on the smallest cardinality of generators of a finite group $G$.
Clearly any finite group $G$
contains an invariable generating set (consider the set of representatives of
each of the conjugacy classes). Moreover for every
$t\in \mathbb N$ there exists a finite (supersoluble) group
$G$ with the property that $d(G)=t$ and $G$ can be generated with
$t$ elements of coprime order (see Proposition \ref{supersoluble}).

\

For general $G$, the bound on $d(G)$ given in Theorem \ref{3}
cannot be improved: there exists a coprimely invariably generated  group $G$
with  $d(G)=3$ (see Proposition \ref{A_5}). However better resuls hold
under additional assumptions. For example, we have a stronger result for
finite soluble groups.

\begin{thm}\label{soluble}
	Let $G$ be a coprimely invariably generated group. If $G$ is soluble,
 then $d(G) \leq 2$.
\end{thm}

A motivation for our interest in coprime invariable generation
is the fact that this property is satisfied by finite
groups without proper subgroups of the same exponent (we will call
these groups {\sl{minimal exponent groups}}). Indeed, assume that
$G$ is a minimal exponent group with $e:= \exp(G)=p_1^{n_1}\cdots p_t^{n_t}$. Then for every $i,$ the group
$G$ contains an element $g_i$ of order $p_i^{n_i}$.
Clearly $\exp \langle g_1^{x_1},\dots,g_t^{x_t} \rangle=e$,
for every $x_1,\dots,x_t\in G$. Hence our assumption that no proper subgroup of
$G$ has exponent $e$ implies that
$G=\langle g_1^{x_1},\dots,g_t^{x_t} \rangle,$ so  $G$ is coprimely invariably generated.
In particular, as a corollary of Theorems~\ref{3} and \ref{soluble}, we deduce
a result already proved in \cite{lms} and \cite{PLN}: a finite group $G$ contains
a 3-generated subgroup $H$ with $\exp(G)=\exp(H)$ and if $G$ is soluble
there exists indeed a 2-generated subgroup $H$ of $G$ with $\exp(G)=\exp(H).$

Notice that the example given in Proposition~\ref{A_5} of  a coprimely invariably generated
group $G$ which is not 2-generated is not minimal exponent. Indeed, the property of
being minimal exponent
is much stronger than coprime invariable generation.

Whereas the bound
$d(G)\leq 3$ in Theorem~\ref{3} cannot be improved, we have no example
of a finite minimal exponent group $G$ which cannot be generated by 2 elements
and the following interesting question is open: {\sl{is it true that any finite
group $G$ contains a 2-generated proper subgroup with the same exponent?}}
We think that the study of coprimely invariably generated groups
could help to answer this question.

The minimal exponent property is not inherited by quotients; conversely, all
epimorphic images of a coprimely invariably generated group (and consequently
of a minimal exponent group) are  coprimely invariably generated. From
this point of view, studying  coprimely invariably generated groups
yields information about quotients of  minimal exponent groups.

\

Another result in this paper concerns the presentation rank of coprimely invariably generated groups.
The \emph{presentation rank} $pr(G)$ of a finite group $G$ is an invariant whose definition comes from the study of relation modules (see \cite{dv-l-L_t} for more details). Let $I_G$ denote the augmentation ideal of $\mathbb Z G$, and $d(I_G)$ the minimal number of elements of $I_G$ needed to generate $I_G$ as a $G$-module, then $d(G)=d(I_G)+pr(G)$ \cite{roggen}. It is known that $pr(G)=0$ for many groups $G$, including all soluble groups, all Frobenius groups and all $2$-generated groups.

\begin{thm}\label{zero-rank}
	Let $G$ be a coprimely invariably generated group. Then $G$ has zero presentation rank.
\end{thm}

As an immediate corollary, we get the following.

\begin{thm}
Let $G$ be a finite group such that no proper subgroup has the same exponent as
 $G$. Then $G$ has zero presentation rank.
\end{thm}

As a further contribution to the understanding of coprimely invariably generated groups, we present the following theorem.

\begin{thm} \label{thm:simpleCIG}
Let $G$ be a finite simple group. Then $G$ is coprimely invariably generated.
\end{thm}

Finally, the following result on conjugacy classes of finite simple groups
may be of independent interest. If $G$ is a group and $\pi = \{p_1, \ldots,
p_k\}$ a set of primes,  then  $|G|_\pi$ denotes the $\pi$-part of
$|G|$ and an element of $G$ whose order is  $p_1^{\alpha_1} \cdot  \cdots \cdot p_k^{\alpha_k}$, for
some
$\alpha_1, \ldots, \alpha_k \in \mathbb{Z}_{\ge 0}$, is a
$\pi$-element.  Notice that the identity is a $\pi$-element.
We let $k_ \pi(G)$ denote the number of conjugacy classes
of $\pi$-elements of $G$.

\begin{thm}\label{k_pi(S)}
Let $S$ be a finite simple group and let $\pi_1, \dots ,\pi_u$ be a partition of $\pi(S)$.  Then
$$\prod_{i=1}^u k_{\pi_i}(S) \leq \frac{|S|}{2|\out S|}.$$
\end{thm}

This paper is structured as follows. In Section~\ref{sec:background} we present
some background information needed for our proofs.
In Section~\ref{sec:ex} we construct two interesting
examples, exploring the necessity and sufficiency of coprime invariable generation in controlling
minimal generation and exponent.
In Section~\ref{sec:soluble} we prove Theorem~\ref{soluble}, then in Section~\ref{sec:3_zero} we prove Theorems~\ref{3} and \ref{zero-rank}. Finally, in Sections~\ref{sec:simpleCIG}
and \ref{sec:k_pi(S)} we prove Theorems~\ref{thm:simpleCIG} and \ref{k_pi(S)}, respectively.

\section{Background material}\label{sec:background}

In this section we introduce primitive monolithic groups and crown-based powers,
and collect some information about their minimal number of generators, and about their presentation rank.

A group $L$ is \emph{primitive monolithic} if $L$ has a unique minimal normal subgroup $A$, and trivial Frattini subgroup.
We define the \emph{crown-based power} of $L$ of \emph{size $t$} to be
\[L_t=\{(l_1,\dots,l_t)\in L^t\mid l_1A=\dots =l_tA\}=A^t \diag(L^t).\]
In \cite{dv-l-L_t} it was proved that, given a finite group $G$, there exist a primitive monolithic group $L$
and a positive integer $t$ such the crown-based power $L_t$ of size $t$ is an epimorphic image
of $G$ and $d(G)=d(L_t) > d(L/\soc (L))$.

The minimal number  of generators of a crown-based power $L_t$ in the case where $A$ is abelian can be computed with the
following formula:
 \begin{thm}\cite[Proposition 6]{DV-L-M} \label{d-abel}
 Let $L$ be a primitive monolithic group with abelian socle $A$, and let $t$ be as above.
Define
$$r_L(A)=\dim_{\End_{L/A}(A)}A  \quad \quad s_L(A)=\dim_{\End_{L/A}(A)} H^1(L/A,A)$$
 and  set $\theta=0$  if $A$ is a trivial $L/A$-module,  and $\theta = 1$ otherwise.
 Then
$$d(L_t)= \max \left( d(L/A), \theta +\left\lceil \frac{t+s_L(A) }{r_L(A)}\right\rceil \right)
$$
where $\lceil x \rceil$ denotes the smallest integer greater or equal to $x$.
\end{thm}

A result of Aschbacher and Guralnick \cite{A-G} assures us that  $s_L(A) < r_L(A)$:
 \begin{thm} \cite{A-G} \label{thm:A-G}
Let $p$ be a prime and $G$ be a finite group. If $A$ is a faithful irreducible $G$-module over $\F_p$, then
$|H^1(G,A)| < |A|$.
\end{thm}

For soluble $G$ we will use the following (the proof can be found in \cite{st}):
 \begin{thm}[Gasch\"{u}tz]\label{H=0}
Let $p$ be a prime. If $G$ is a finite $p$-soluble group and $A$ is a faithful irreducible $G$-module over $\F_p$, then $|H^1(G,A)| =0.$
\end{thm}

When $A$ is non-abelian,  $d(L_t)$ can be evaluated using the following,
where $P_{L, A}(k)$ denotes the conditional probability that $k$ randomly chosen elements of $L$ generate $L$, given that they project onto generators for $L/A$.

\begin{thm}\cite[Theorem 2.7]{dv-l-L_t}\label{f_L}
Let $L$ be a monolithic primitive group with  non-abelian socle $A$, and let $d \geq d(L)$.
Then $d(L_t) \leq d $ if and only if
\[t \leq \frac{P_{L,A}(d) |A|^d}{|C_{\aut L}(L/A)|}.\]
\end{thm}

Bounds on $P_{L,A}(d)$ were studied in \cite{PLN} and \cite{nina},   achieving  the strong result:
\begin{thm}\cite{PLN} \label{1/2}
Let $L$ be a primitive monolithic group with socle $A$. Then $P_{L,A}(d)\ge 1/2.$
\end{thm}


We finish this introductory section with a result on presentation rank.

\begin{thm}\label{pr}
Let $G$ be a finite group and let $L_t$ be a crown based power of a primitive monolithic group
$L$  such that  $L_t$ is a homomorphic image of $G$ and
 $d(G)=d(L_t)>d(L/\soc(L))$.
 If $\soc(L)$ is abelian, then $pr(G)=0$.
\end{thm}
\begin{proof}
For an irreducible $G$-module $M$,
 we set
$$r_G(M)=\dim_{\End_G(M)}M  \quad \quad s_G(M)=\dim_{\End_G(M)} H^1(G,M)$$
 and define
$$h_{G}(M)= \theta+ \left\lceil    \frac{s_G(M)}{r_G(M)}\right\rceil
$$
where $\theta=0$ if $M$ is a trivial and $\theta = 1$ otherwise.

Assume that 
$A=\soc(L)$ is abelian. Let  $\delta_G(A)$ be the largest integer $k$ such that the
crown based power $L_{k}$
 is a homomorphic image of $G$, and note that
$$d(L_{\delta_G(A)})=d(L_t)=d(G).$$
By \cite[Proposition 9]{crowns}, the integer
$\delta_G(A)$ is the number of complemented chief factors
$G$-isomorphic to $A$ in any chief series of $G$. Since
$$r_G(A)=r_{L}(A)$$
and
$$s_G(A)=\dim_{\End_G(A)} H^1(G,A)= \delta_G(A)+ \dim_{\End_{L/A}(A)} H^1(L/A,A) $$
 (see e.g. [1.2] in \cite{andrea2}),
it follows that
\begin{equation*}
h_G(A)= 
\theta+ \left\lceil    \frac{\delta_G(A)+s_{L}(A)}{r_G(A)}\right\rceil.
\end{equation*}
By Theorem \ref{d-abel}
we conclude that
\[d(G)=d(L_{\delta_G(A)})=h_G(A).\]
 By a result of Cossey, Gruenberg and Kov\'acs \cite[Theorem 3]{pr}
$$d(I_G)= \max \{ h_G(M)\mid  M  \mbox{ an irreducible }  G\textrm{-module} \}$$
 thus, in particular, $d(I_G) \geq h_G(A)=d(G)$.
  Since $d(I_G)\leq d(G)$, we have an equality, hence $pr(G)=0$.
  \end{proof}


\section{Examples}\label{sec:ex}

In this section, we start by constructing a group that shows that the bound given in Theorem \ref{3} cannot be improved.
The same group provides an example
  of a coprimely invariably generated group which is not minimal-exponent.
 We then construct a family of examples which demonstrate that the property of coprime generation alone is not enough to
constrain the minimal number of generators of a finite group.

\begin{prop}\label{A_5}
Let $L = \mathrm{ASL}_{2}(4) = \F_4^2 \rtimes \slin{2}{4} = V \rtimes \slin{2}{4}$,
and let $G$ be the crown-based power $L_2$ of $L$.
Then $G$ is coprimely invariably generated and $d(G)=3$.
 Moreover,  $G$ has a proper subgroup with the same exponent.
\end{prop}
\begin{proof}
Note first that $|H^1(\slin{2}{4},V)|=4$. Thus we may use
Theorem~\ref{d-abel} with $t = 2$, $\theta = 1$, $r_{L}(V) = 2$ and $s_{L}(V) = 1$ to
see that $d(G)=3$.

Let us now show that $G$ is coprimely invariably generated.
Let $z=(e_1,e_2)\in V\times V$, where $e_1$ and $e_2$ are linearly
independent elements of $V$ and note that, with this assumption on $e_1$ and $e_2$, the normal closure $\langle z   \rangle^G$ 
is the whole of $V^2$. Choose
 $x \in \slin{2}{4}$ of order $3$, and $y \in \slin{2}{4}$ of order $5$.
For any $a,b,c \in G$ and for $H= \langle x^a,y^b,z^c \rangle$, the quotient
$HV^2/V^2 \cong \slin{2}{4}$, hence $\langle z^c \rangle^H=\langle z^c \rangle^G=
\langle z   \rangle^G=V^2$.  Therefore $H=HV^2=G$ and we conclude that $G$ is invariably generated by $x,y,z$.

Finally, the subgroup $\{(l,l) \in L^2\mid l\in L \}$ is a proper subgroup of $G$ with the same exponent.
\end{proof}

\begin{prop}\label{supersoluble} For any $t\in \mathbb N$ there exists a finite supersoluble group
$G$ such that $G$ can be coprimely generated with $d(G) = t$ elements.
\end{prop}

\begin{proof}Let $n=p_1\cdots p_t$ be the product of the first $t$ prime integers and let $p$ be a prime
such that $n$ divides $p-1$ (the prime $p$
exists by Dirichlet's theorem). The cyclic group $C = \cyc{n}$ has a
fixed point free multiplicative action on $V=\F_p$;
set $L$ to be the monolithic group $V \rtimes C$.
Let $G$ be the crown-based power $L_t$, then  $d(G)=t+1$  by Theorem~\ref{d-abel}.

Consider a generating set $\{x_1,\dots,x_{t+1}\}$ of $C$ with $|x_i|=p_i$ if $i\leq t$
and $x_{t+1}=1.$
A well-known theorem of W. Gasch\"utz  \cite{g1} states that
 if $F$ is a free group with $n$ generators, $H$ is a group with $n$ generators, and $N$ is a finite normal subgroup of $H$, then every homomorphism of $F$ onto $H/N$ is induced by a homomorphism of $F$ onto $H$.
It follows that there exist $w_1,\dots,w_{t+1}$ such that
$G=\langle x_1w_1,\dots,x_{t+1}w_{t+1}\rangle.$ Clearly $|x_{t+1}w_{t+1}|=p$; on the other hand
if $i\leq t$, then $C_{x_i}(V)=\{0\}$, and this implies that $|x_iw_i|=|x_i|=p_i.$
\end{proof}

%
%


\section{Proof of Theorem \ref{soluble} }\label{sec:soluble}

\begin{thm}\label{abcase} Let $L=A \rtimes H$ be a primitive monolithic group with abelian socle $A$ and let  $t \in \mathbb{N}$.
 If  $L_t$
 is coprimely invariably generated, then $$t\leq \dim_{\End_H(A)}A.$$
\end{thm}
\begin{proof} Let $A$ be a $p$-group and set $G=L_t$.
Assume that $\{g_1,\dots,g_u\}$ is a  set of pairwise coprime elements that invariably generate $G$
where $g_i$ is a $p'$-element for every $i \neq 1$. Set $V=A^t$.

Note that, if $|g_i|$ is coprime to $p$ and $g_i=vh$ where $v \in V$ and $h \in H$, then $g_i$ is conjugate to an element of $\langle h \rangle$, since  $\langle h \rangle$ is a Hall $p'$-subgroup of $V \langle h \rangle$; in particular  $g_i$ is conjugate to an element of $H$.
 Therefore,
as  $\{g_1, g_2, \dots , g_u\}$ invariably generates $G$, by
taking suitable conjugates of $g_2, \dots , g_u$, we can assume that $g_2, \dots , g_u \in H$.

Consider $g_1=vh$, where $v \in V$ and
 $h \in H$, and set $K=\langle  h, g_2,\dots,g_u \rangle$.
Since $KV=G=HV$ and $K \leq H$, we deduce that $K=H = \langle h, g_2, \ldots, g_u \rangle$.
Therefore,
\[  G= \langle  vh, g_2,\dots,g_u \rangle \leq \langle  v, h, g_2,\dots,g_u \rangle
 \leq \langle v \rangle ^H H \]
 hence $G=\langle v \rangle ^H H $ and $\langle v \rangle ^H=V$,
 that is,
  $v$ is a cyclic generator for the $\F_pH$-module $V=A^t$.
  Let $v=(v_1, \ldots ,v_t)$.
 Switching to additive notation,  the fact that $v$ is a cyclic generator for the $\F_pH$-module $V$ implies   that
 the elements $v_1, v_2, \ldots , v_t$ are linearly independent elements of the  $\End_H(A)$-vector space $A$.
 In particular $t\leq \dim_{\End_H(A)}A$, as required.
\end{proof}

\begin{proof}[Proof of Theorem \ref{soluble}]
	Let $G$ be a soluble, coprimely invariably generated group. Let $L_t$ be a crown based power such that $L_t$ is a homomorphic image of $G$ and $d(G)=d(L_t)> d(L/A)$. Then $L_t$ is coprimely invariably generated and $L$ has abelian socle.
 Let 	 $r_L(A)$ and $s_L(A)$ be as in Theorem~\ref{d-abel}.

   Since $L$ is soluble, we see from Theorem~\ref{H=0} that $s_L(A)=0$.
Moreover  Theorem  \ref{abcase} implies that $t \le r_L(A)$, and thus  	
$\lceil (t+s_L(A)) /{r_L(A)}\rceil =1$.
 As $d(L_t)>d(L/A)$, by Theorem \ref{d-abel} we conclude that
	\[d(L_t)= \theta +\left\lceil\frac{ t+s_L(A) }{r_L(A)}\right\rceil \le 2,\]
as required.
\end{proof}


\section{Proof of Theorems \ref{3} and \ref{zero-rank}}\label{sec:3_zero}

Let $L$ be a finite monolithic group whose socle $A$ is non-abelian and let
 $\pi$ be a set of primes. For every  $l \in L$,
 define $a_l$ to be the number of $A$-conjugacy classes of  $\pi$-elements $L$ which are contained in $lA$.
Then set
\[a_{\pi}= \max \{a_l \mid l \in L \} .\]
 As usual, for an integer $n$, the set of prime divisors of $n$ is denoted $\pi(n)$.

 \begin{thm}\label{t} Let $L$ be a finite monolithic group whose socle $A$ is non-abelian and let $t$ be a positive integer.
 If the set  $\{g_1,\dots,g_u\}$ invariably generates $L_t$,  
 then
$t\leq \prod_i a_{\pi(|g_i|)}.$
\end{thm}

\begin{proof}
Assume that
 $\{g_1,\dots,g_u\}$ invariably generates $L_t$,  and set $\pi(|g_i|)=\pi_i$, for every $i$.
 Note that, by the definition of $L_t$,   $g_i=(x_{i1},\dots,x_{it})$ where $x_{i1}, \dots,x_{it}$ belong to the same coset $l_iA$ for some
 $l_i \in L;$ in particular $x_{i1}, \dots,x_{it}$ are $\pi_i$-elements of $l_iA$.

If there exist $r$ and $s$ such that $x_{is}=x_{ir}^y$
for some $y\in A$, then by replacing $g_i$ by a suitable conjugate we can assume that  $x_{is}=x_{ir}$ (more precisely, we take the conjugate of $g_i$ by the element $\overline{y}=(1,\ldots,y,\ldots,1)\in L_t$, where $y$ is in the $r$-th position).
Let $a=\prod_{i}a_{\pi_i}$. If $t>a$,
 then it follows from the definition of $a_{\pi}$ that there exist $r,s\in\{1,\ldots,t\}$ with $r\ne s$ such that $x_{ir}=x_{is}$ for every $i\in\{1,\dots,u\}$.
But then $\langle g_1,\dots,g_u\rangle \leq \{(l_1,\ldots,l_t)\in L_t \mid l_r=l_s\}$ which is a proper subgroup of $L_t$, a contradiction.
\end{proof}

\begin{lemma}\label{a_l}
Let $L$   be a monolithic primitive group with non-abelian socle $A$ and
let $\pi$ be a set of primes. 
Then \[a_{\pi} \le k_ {\pi}(A) .\]
\end{lemma}
\begin{proof}
Let $l$ be a $\pi$-element  of $L$ such that $a_l=a_{\pi}$.
Set $X=\langle l \rangle A$. Let $x \in lA$. Since $X/A= \langle xA \rangle$, we have $X=AC_X(x)$ whence every $X$-conjugacy class  in $lA$ is a single $A$-orbit.  In particular $a_l$ coincides with the number of $X$-conjugacy classes of $\pi$-elements in the coset $lA$.

By \cite[Theorem 1.6]{centralizer},
 $a_l$ is precisely the number of $A$-conjugacy classes of $\pi$-elements in $A$ which are invariant under $X$, whence $ a_l \le k_ {\pi}(A)$.
\end{proof}

\begin{lemma}\label{prod}
Let $L$   be a monolithic primitive group with non-abelian socle $A=S^n$, and
let $\pi_1, \dots ,\pi_u$ be disjoint sets of primes. Then
\[\prod_{i=1}^u a_{\pi_i} \le \frac{|A|}{2n|\out S|} .\]
\end{lemma}
\begin{proof} 
 By Lemma \ref{a_l}, we may bound $a_{\pi_i} \le k_{\pi_i}(A)$ for all $i$.
As $A=S^n$, we get  $k_{\pi_i}(A) = k_{\pi_i}(S)^n$.
Now consider a partition $\tilde \pi_1, \dots, \tilde \pi_u$ of $\pi(|S|)$ with the property that $\tilde \pi_i \supset \pi_i \cap \pi(|S|):$
 clearly $k_{\pi_i}(S) \leq k_{\tilde \pi_i}(S)$.
It follows from  Theorem~\ref{k_pi(S)}  (whose proof is in Section~\ref{sec:k_pi(S)}) that
$$ \prod_{i=1}^u k_{\tilde \pi_i}(S) \leq \frac{|S|}{2|\out S|}.$$
Therefore
\[\prod_{i=1}^u  a_{\pi_i} \le \prod_{i=1}^u  k_{\pi_i}(A)
  = \prod_{i=1}^u  k_{\pi_i}(S)^n
\le \prod_{i=1}^u  k_{\tilde \pi_i}(S)^n
  \leq \frac{|S|^n}{2^n|\out S|^n}
 \le \frac{|A|}{2n|\out S|} \]
as required.
\end{proof}

\begin{lemma}\label{d=2}
Let $L$ be a monolithic primitive group with non-abelian socle $A=S^n$.
 If $L_ t$ is minimally $d$-generated
(i.e. $d(L_{ t}/N )< d(L_ t)=d$ for every $ 1 \neq N \lhd L_{ t}$) and
\[ t \le  \frac{| A|}{2n |\out S|}\]
then $d = 2$ (and $ t=1$).
\end{lemma}
\begin{proof}
Set $d_L =d(L)$ and note that $d_L \ge 2$ since $L$ has non-abelian socle. Let  $X$ be the subgroup of $\aut S$ induced by the conjugation action of $N_G(S_1)$ on the first
factor $S_1$ of $A=S_1\times \dots \times S_n$, with $S\cong S_i$ for each $1\leq i \leq n$. As in the proof of Lemma 1 in \cite{dv-l-non-zero-pr},
\[|C_{\aut A}(L/A)| \leq n |S|^{n-1}|C_{\aut S}(X/S)| \]
and therefore
\[|C_{\aut A}(L/A)| \leq n  |S|^{n-1} |\aut S| = n |A| |\out S|.\]
By  Theorem  \ref{1/2}, $P_{L,A}(d_L)\ge 1/2.$
So the assumptions give that
\[ t \le \frac{1}{2} \frac{| A|}{n |\out S|} \le \frac{P_{L,A}(d_L)|A|^2}{n |A||\out S|} \le \frac{P_{L,A}(d_L)|A|^{d_L}}{|C_{\aut A}(L/A)| }.\]
By Theorem \ref{f_L} this implies that $d=d(L_ t)=d_L$.  As $L_ t$ is minimally $d$-generated, it follows that $ t=1$; in particular, $L$ is minimally $d$-generated.  Now, by the main theorem in \cite{meneg}, $d(L)= \max \{ 2, d(L/A) \}$,  and again by minimality, we conclude that $d=d(L)=2$.
\end{proof}

\begin{proof}[Proof of Theorem \ref{3}]
Let $G$ be a coprimely invariably generated group and let $d=d(G)$. As remarked in Section~\ref{sec:background}, there exists a monolithic primitive group $L$ with socle $A$ and an integer $t$, such that $L_t$ is a quotient of $G$ and $d=d(L_t)> d(L/A)$. Moreover $L_t$ is coprimely invariably generated.

 If $A$ is abelian,  then we can apply Theorem \ref{d-abel}: since $ d(L_t)> d(L/A)$ and, by  Theorems  \ref{thm:A-G} and \ref{abcase},   $s_L(A) < r_L(A)$ and $t \le r_L(A)$,  it follows that
 \[d(G) = d(L_t)= \theta +\left\lceil\frac{ t+s_L(A)}{r_L(A)} \right\rceil \le \theta + 2 \le 3.\]

If $A$ is non-abelian and $\{g_1,\dots,g_u\}$ are coprime invariable generators of $L_t$,
 then by Theorem \ref{t},
$t\leq \prod_{i=1}^u a_{\pi(|g_i|)}.$
Then by Lemma \ref{prod}
\[\prod_{i=1}^u a_{\pi(|g_i|)} \le \frac{|A|}{2n|\out S|} .\]
 Thus
 \[ t \le  \frac{|A|}{2n|\out S|} \]
  and by Lemma \ref{d=2} we conclude that $d(G)=d(L_t)=2$.
\end{proof}

\begin{proof}[Proof of Theorem \ref{zero-rank}]
Let $G$ be a coprimely invariably generated group.
Assume, by way of contradiction,
that $pr(G)>0$. Let $L$ and $t \in \N$ be such that $L$ is a monolithic primitive group
with socle $A$ and $L_ t$ is a homomorphic image of $G$,   with $d(L_t)=d(G)=d$ and $d>d(L/N)$.

 If $A$ is abelian, then $pr(G)=0$ by
 Theorem \ref{pr}, a contradiction.
If  $A$ is non-abelian, then arguing
as in the proof of the non-abelian case of Theorem \ref{3},  we conclude that $d=d(L_t)=2.$
 Thus again $pr(G) = 0$. 
\end{proof}

\section{Proof of Theorem~\ref{thm:simpleCIG}}\label{sec:simpleCIG}

In the following proof, by $[a, b]$ we denote the lowest common multiple of integers $a$ and $b$.

\begin{proof}[Proof of Theorem~\ref{thm:simpleCIG}]
We make use of the invariable generators given in \cite{ig}, where it is proved that every finite simple group is invariably generated by two elements. For the classical groups, the orders given in \cite{ig} are for the quasisimple groups, so we must adjust their values to get coprime projective orders.

 For the alternating groups, the generators given in \cite[Proof of Lemma 5.2]{ig} are of coprime orders.

 For the special linear groups in dimension $n \ge 3$, the invariable generators in \cite{ig} have orders $(q^n-1)/((q-1, n)(q-1))$ and $(q^{n-1} - 1)/(q-1)$, which are coprime. The given generators for $\psl{2}{q}$ are also always of coprime order. For the unitary groups and the orthogonal groups other than $\mathrm{P \Omega}_{4k+2}^{-}(q)$ with $q$ odd  and $\om{8}{+}{q}$ with $q \le 3$, the given generators are coprime. For $\mathrm{P \Omega}_{4k+2}^{-}(q)$ it suffices to take the square of the second generator in \cite{ig} to produce coprime invariable generators.

For the symplectic groups in dimension $2m \ge 4$, the given generators have orders $(q^m + 1)/(q-1, 2)$ and $[q^{m-1}+1, q+1]$,  so when $m$ is even as the corresponding elements of the simple group are coprime. When $m$ is odd we choose three elements: one of order $(q^m + 1)/(2, q-1)$, one of order $(q^{m-1} + 1)/2$, and one of order  $(q^m - 1)/2$. These are coprime, and it follows from \cite[Theorem 1.1]{MSW} that these elements  invariably generate $\mathrm{PSp}_{2m}(q)$.

Of the classical groups, this leaves only $\om{8}{+}{2}$ and $\mathrm{P \Omega}_{8}^{+}(3)$. For $\om{8}{+}{2}$,
choose an element $a$ from class \texttt{5A}, an element $b$ from class \texttt{7A} and an element $c$ from class \texttt{9A}. The only maximal subgroup of $\om{8}{+}{2}$ to contain elements of all three of these orders is $\symp{6}{2}$, and $\symp{6}{2}$ has no maximal subgroups that contain elements of all three of these orders, so if $\langle a^x, b^y, c^z \rangle = H \neq \om{8}{+}{2}$ then $H \cong \symp{6}{2}$, and $H$ contains \texttt{5A} elements of $\om{8}{+}{2}$. The outer automorphism group of $\om{8}{+}{2}$ acts on the three classes of $\symp{6}{2}$ in the same way as it acts on the three classes of elements of order $5$. Thus specifying that $H$ contains \texttt{5A} elements tells us which $\om{8}{+}{2}$-conjugacy class of groups $\symp{6}{2}$ we have, and in particular without loss of generality $H = \langle a, b, c \rangle$. Such a group contains no elements from (a fixed) one of classes \texttt{2C} or \texttt{2D}, so we let $d$ be an element from this class. Then $a, b, c$ and $d$ coprimely invariably generate $\om{8}{+}{2}$.

For $G = \mathrm{P \Omega}_{8}^{+}(3)$, we first note that $G$ contains three classes of elements of order $5$,
two of order $13$ (one of which contains powers of the other), and only one of elements of
order $7$. So we let $a \in \texttt{7A}$, $b \in \texttt{13A}$ and $c \in \texttt{5A}$.
Order considerations show that the only possible maximal subgroup to contain
$\langle a^x, b^y, c^z \rangle$ is $\om{7}{}{3}$, and that given $a^x, b^y, c^z$, they
are contained in at most one copy of $\om{7}{}{3}$. There are six classes of
groups $\om{7}{}{3}$ in $G$, with stabiliser $\mathrm{D}_8$, and these classes are cycled in
two 3-cycles by the triality automorphism, which also permutes the three classes of elements of order $5$ in $G$. Thus there are most two $G$-conjugacy classes of
groups $\om{7}{}{3}$ that are generated by $\langle a^x, b^y, c^z\rangle$ as $x, y, z$
vary. Now, $\om{7}{}{3}$ contains four conjugacy elements of order $9$, which form three
orbits under $\Aut \om{7}{}{3}$. Conversely, $\om{8}{+}{3}$ contains $14$ conjugacy classes of elements of order $9$, forming orbits of length $6$, $4$ and $4$. Consider the orbit of length $6$. The two $G$-conjugacy classes of groups $\om{7}{}{3}$ intersect at most four of these classes, so let $d$ be an element of order $9$ in one of the remaining two classes. Then $a, b, c, d$ coprimely invariably generate $G$.


For all of the exceptional groups except $\cE{7}{q}$,
 the invariable generators given in \cite{ig} are coprime.
Thus we need only consider $\cE{7}{q}$. By \cite[Table 6]{GM}, elements of order $(q+1)(q^6 - q^3 + 1)/(2, q-1)$ are contained only in a copy of $\tE{q}_{sc}.\mathrm{D}_{q+1}$. Since the order of $\cE{7}{q}$ is divisible by $q^{14}-1$, we may find an element of order a Zsigmondy prime for $q^{14}-1$ in $\cE{7}{q}$. Such a prime does not divide the order of $\tE{q}$ or $q+1$, so gives a pair of invariable generators for $\cE{7}{q}$.

For the sporadics and the Tits group, \cite[Table 9]{GM} lists carefully chosen conjugacy classes of elements of the sporadics groups, together with a complete list of the maximal subgroups containing those conjugacy classes. It suffices to check that in each case there exists an element
of order coprime to the given one that  lies in none of the listed maximal subgroups.
\end{proof}

\section{Proof of Theorem \ref{k_pi(S)} } \label{sec:k_pi(S)}

In this section we prove Theorem~\ref{k_pi(S)}. First, we need a preliminary lemma.

\begin{lemma}\label{k}
 Assume that $G$ is a finite group and let $\pi\subseteq \pi(G).$
Then $k_ \pi(G)\leq  |G|_\pi$. In particular if  $\pi=\{p\}\cup \tilde \pi$, then
    $k_ \pi(G) \leq k_ p(G) \cdot |G|_{\tilde \pi}.$
\end{lemma}
\begin{proof}
We prove that  $k_ \pi(G)\leq |G|_\pi$ by
 induction on $|\pi|.$
The case $|\pi|=1$ is an immediate consequence of
the Sylow Theorems. Assume $\pi=\{p\}\cup \tilde \pi.$ Let $g$ by a $\pi$-element of $G$;
we may write $g=ab$ where $a$ is a $p$-element and $b$ is a $\tilde \pi$-element and both are powers of $g$.
Up to conjugacy
 we have at most $k_ p(G)$ choices for $a$. For a fixed choice of $a$ we have
to count the number of $b$. Notice that $b\in H=C_G(a).$ Moreover if $b_1$ and $b_2$ are conjugate
in $H$ then $ab_1$ and $ab_2$ are conjugate in $G$. Hence the number of choices of $b$ is bounded
by the number of conjugacy classes of $\tilde \pi$ elements in $H$, and by induction this number
is at most $|H|_{\tilde \pi}\leq |G|_{\tilde \pi}.$ Thus $k_ \pi(G)\leq  |G|_\pi$ as required.

By the same argument, we now have that
 \[k_ \pi(G)  \leq k_ p(G) k_{\tilde \pi}(H) \leq k_ p(G) |H|_{\tilde \pi} \leq  k_ p(G)  |G|_{\tilde \pi}.\]
\end{proof}

We in fact prove a slightly stronger version of Theorem~\ref{k_pi(S)}, which we state now. Let $\mathcal{S} = \{\alt{n} \ : \ n \le 7\} \cup \{\lin{2}{q} \ : \ q \in \{7, 8, 11, 27\}\} \cup \{\lin{3} {4}\}$.

\begin{thm}
Let $S$ be a finite simple group and let $\pi_1, \dots ,\pi_u$ be a partition of $\pi(S)$.  Then
$$\prod_{i=1}^u k_{\pi_i}(S) \leq \frac{|S|}{2|\out S|}.$$
Furthermore, if $S \not\in \mathcal{S}$, then there exists a prime $p$ dividing $|S|$ such that
\[k_p(S) \le \frac{|S|_p}{2|\out S|}.\]
\end{thm}

\begin{proof}
For groups in $\mathcal{S}$,  this is a direct calculation using their conjugacy classes.
For the remaining groups, the first claim follows from the second and Lemma \ref{k}.
The alternating case is considered in Lemma~\ref{A_n}, below.
The linear and unitary groups and the symplectic and orthogonal groups are dealt with in Lemmas~\ref{lin_u} and \ref{sp_ort}, respectively. The exceptional case is completed in Lemma~\ref{exceps}.  For the sporadics, this is a straightforward exercise, using \cite{ATLAS}.
\end{proof}


\begin{lemma}\label{A_n}
Let $S = \alt{n}$ for some $n \geq 7$. Then
there exists a prime $r$ dividing $|S|$ such that
$S$ has at most one conjugacy class of nontrivial $r$-elements. Furthermore, if $n \ge 8$ then there exists a prime $p$ dividing $|S|$ such that
\[k_p(S) \le \frac{|S|_p}{2|\out S|}.\]
\end{lemma}

\begin{proof}
First let $k = \lfloor n/2 \rfloor$.
Then Bertrand's postulate states that for $k \geq 4$, there exists a prime $r$ such that
$k \le n/2 < r < 2k - 2 \in \{n-2, n-3\}$, so the first claim follows (after verifying that $r = 5$ works when $n = 7$).

As for the second claim, note that $|\out S| = 2$.
For $n = 8$, we use $k_2(S) = 5$ whilst $|S|_2 = 2^6$. For $n = 9$, we use $k_3(S) = 6$ whilst $|S|_3 = 3^4$.
For $n \in \{10, 11, 12, 13\}$ we use $k_{5}(S) = 3$.
We may therefore assume that $n \ge 14$ and $n - 2 > p= r \geq 11$.
Thus $k_p(S) = 2$, whilst $\frac{|S|_p}{2|\out S|} \geq 11/4 > 2$, so the result follows.
\end{proof}

\begin{lemma}\label{lin_u}
Let $S \cong \lin{n}{p^e}, \uni{n}{p^e}$ be simple, and assume that $S \not \in \{\lin{2}{q} \ : \ q \in \{4, 5, 7, 8, 9, 11, 27\}\} \cup \{\lin{3}{4}\}$.
Then
\[k_p(S) \le \frac{|S|_p}{2|\out S|}.\]
\end{lemma}

\begin{proof}
By \cite[Lemma 1.4]{LiebeckPyber}, $k_p(S) \le np(n) + 1$, where $p(n)$ is the partition function of $n$. Since $p(n) =  p(n-1) + p(n-2) - p(n-5)  - p(n-7) + \cdots$, where the sum is over the pentagonal numbers less than $n$ and the sign of the $k$th term is $(-1)^{\lfloor (k-1)/2 \rfloor}$, we may bound $np(n)  + 1 \leq n2^{n}$.

First suppose that $n = 2$, so that $|S|_p = q$. Then without loss of generality $S \cong \lin{2}{p^e}$. Here $k_p(S) = 2$ for $p = 2$, and $3$ for $p$ odd, whilst $|\out S|$ is $e$ for $p = 2$ and $2e$ for $p$ odd. Thus for $p = 2$  we must check that $2^e \geq 2 \cdot 2e$, which holds for all $e \ge 4$. For $p$ odd we require $p^e \geq 12e$, which clearly holds for all $e$ when $p \ge 13$. If $p = 3$ this yields $e \ge 4$, and when $5 \le p \le 11$ this yields $e \ge 2$.

Next suppose that $n = 3$, so that $|S|_p = q^3$. Suppose first that $S \cong \lin{3}{p^e}$. If $q \equiv 1 \bmod 3$ then $k_p(S) = 5$ and $|\out S| = 6e$, so we require $p^{3e} \geq 60e$, which holds for all such $q > 4$.
If $q \equiv 0, 2 \bmod 3$ then $k_p(S) = 3$ and $|\out S| = 2e$, so we require $p^{3e} \geq 12e$, which holds for all $q > 2$ (but recall that $S \not\cong \lin{3}{2} \cong \lin{2}{7}$). Suppose next that $S \cong \uni{3}{p^e}$. In this case, if $q \equiv 2 \bmod 3$ then $k_p(S) = 5$, whilst if $q \equiv 0, 1 \bmod{3}$ then $k_p(S) = 3$. Since $\uni{3}{2}$ is not simple, and $|\out S| = (3, q+1)\cdot 2e$, the result follows by a similar calculation to that for $\lin{3}{q}$.

We now consider the general case. We bound $k_p$ by $np(n) + 1 \le n2^{n}$,
 whilst the order of a Sylow $p$-subgroup of $S$ is $q^{n(n-1)/2}$ and $$|\Out S| \leq 2(q-1) \log_p q < q^2.$$
If $n2^{n} \geq q^{n^2/2 - n/2 - 2}/2$ then $(n, q) \in \{(4, 2), (4, 3), (5, 2)\}$. In fact $k_2(\lin{4}{2}) = 5 < 2^6/4$, whilst $k_3(\lin{4}{3}) = 7 < 3^6/8$ and $k_2(\lin{5}{2}) = 7 < 2^{10}/4$,  so the result follows.
\end{proof}

\begin{lemma}\label{sp_ort}
Let $S$ be a simple symplectic or orthogonal group, of rank $n$ over $\F_{p^e}$.
Then \[k_p(S) \le \frac{|S|_p}{2|\out S|}.\]
\end{lemma}

\begin{proof}
Here $|S|_p \geq q^{n^2 - n}$ and $|\Out S| \leq 2(q-1, 2)^2 \log_p q$, which is less than $q^2$ for all $q$.
By \cite[Lemmas 1.4 and 1.5]{LiebeckPyber}
if $S$ is symplectic then $$k_p(S) \leq p(2n)2^{(2n)^{1/2}} < 6^n$$ (where $p(n)$ is the partition function of $n$),
whilst if $S$ is orthogonal then $$k_p(S) \leq 2(n,2)p(2n+1)2^{(2n+1)^{1/2}} < 6^n.$$

If $n = 2$ then $q  > 2$ and $S$ is symplectic, so that $|S|_p = q^4$ and $|\Out S| = 2(q, 2)\log_p q$, whilst $k_p(S) \le 7$, so the result follows for all $q$. 

If $n = 3$ then $k_p(S) \leq 60, 187$ for $S$ symplectic or orthogonal, respectively,
so the result is immediate for $q \ge 5$, and for the remaining $q$ we check that in fact $k_p(S) \le 16$.

If $n = 4$ then $k_p(S) \le 156$ for $S$ symplectic and $960$ for $S$ orthogonal, so the result is immediate for $q \ge 7$. For $2 \le q \le 5$ we verify that in fact $k_p(S) \le 81$, which completes the proof.

If $n \ge 5$ the result follows immediately from the $6^n$ bounds, for all $q$.
\end{proof}

\begin{lemma}\label{exceps}
Let $S \cong {}^rX_l(p^e)$ be a simple group of exceptional type. Then \[k_p(S) \le \frac{|S|_p}{2|\out S|}.\]
\end{lemma}

\begin{proof}
We use the results cited in \cite[Proof of Lemma 1.5]{LiebeckPyber} to bound $k_p(S)$ for each family. Let $q = p^e$.

If $S \cong \cF{q}, \cE{6}{q}, \tE{q}, \cE{7}{q}, \cE{8}{q}$, then $|S|_p \geq q^{24}$ and $|\Out S| \leq 6 \log_p q < q^3$, whilst $k_p(S) \le 202$ so the result is clear.

If $S \cong \cG{q}$ then $|S|_p = q^6$ and $|\Out S| \le 2 \log_p q < q$, whilst $k_p(S) \le 9$.
If $S \cong \tB{q}$ then $|S|_p = q^2$ and $|\Out S|  = \log_p q$, whilst $k_p(S) = 4$, so the result holds for all $q > 2$, however $\tB{2}$ is not simple.  If $S \cong \tD{4}{q}$ then $|S|_p = q^{12}$ and $|\Out S| = 3 \log_p q < q^2$, whilst $k_p(S)  \le 8$, so the result is clear.
 If $S \cong \tG{q}$ then $q \ge 27$ with $|S|_p = q^3$ and $|\Out S| = \log_p q$, whilst $k_p(S) \le 10$, so the result holds for all $q$. Finally, if $S \cong \tF{q}$ then $|S|_p = q^{12}$ and $|\Out S| = \log_p q$, whilst $k_p(S) < 35$.
\end{proof}


\begin{thebibliography}{10}
\bibliographystyle{amsplain}



\bibitem{A-G} M. Ascbacher \& R. Guralnick. Some applications of the first cohomology group.
 {\em J. Algebra} {\bf 90} (1984) 446--460.

\bibitem{ATLAS}
J.H.~Conway, R.T.~Curtis, S.P.~Norton, R.A.~Parker \& R.A.~Wilson.
\newblock {\em An $\mathbb{ATLAS}$ of Finite Groups.}
\newblock Clarendon Press, Oxford, 1985; reprinted with corrections 2003.
%

\bibitem{pr}
J. Cossey,  K. W. Gruenberg \& L. G. Kov\'acs.
The presentation rank of a direct product of finite groups.
\textit{J. Algebra} {\bf 28} (1974) 597--603.



\bibitem{dv-l-L_t}
{ F. Dalla Volta \& A. Lucchini}.
Finite groups that need more generators than any proper quotient.
{\em J.~Austral.~Math.~Soc.~Ser.~A}~{\bf 64:1} (1998) 82--91.

\bibitem{dv-l-non-zero-pr}
{ F. Dalla Volta \& A. Lucchini.}
The smallest group with non-zero presentation rank. {\em J. Group Theory} {\bf 2:2} (1999) 147-–155.

\bibitem{DV-L-M}
{F. Dalla Volta, A. Lucchini \& F. Morini}.
On the probability of generating a  minimal $d$-generated group.
{\em J. Aust. Math. Soc.} {\bf 71:2} (2001) 177--185.
%


\bibitem{crowns}
 E. Detomi \& A. Lucchini.
  Crowns and factorization of the probabilistic zeta function of a finite group.
  \textit{J. Algebra} {\bf 265} (2003) 651--668.
%


\bibitem{PLN}
E. Detomi \& A.~Lucchini.  {Probabilistic generation of finite groups
  with a unique minimal normal subgroup}. {\em J. London Math.~Soc.} {\bf 87:3} (2013)
 689--706.

\bibitem{dm}
John~D. Dixon \& Brian Mortimer. \emph{Permutation groups}.
Graduate Texts in
Mathematics, vol. 163, Springer-Verlag, New York, 1996.

\bibitem{dix2} J.D. Dixon. Random sets which invariably generate the symmetric group. \textit{Discrete Math.} {\bf 105} No.1-3 (1992) 25--39.

\bibitem{centralizer} J. Fulman \& R. Guralnick.
Bounds on the number and sizes of conjugacy classes in finite
Chevalley groups with applications to derangements.
{\em Trans.~Amer.~Math.~Soc.} {\bf 364:6} (2012) 3023--3070.

\bibitem{g1}
{W. Gasch\"utz}. Zu einem von B.H. und H. Neumann gestellten Problem. {\it Math.  Nachr.}  {\bf {14}} (1955) 249--252.




\bibitem{GM}
R. Guralnick \& G. Malle.
Products of conjugacy classes and fixed point spaces.
\textit{J. Amer.~Math.~Soc.} {\bf 25:1} (2011) 77--121.

\bibitem{ig} W.M. Kantor, A. Lubotzky \&  A. Shalev.
Invariable generation and the Chebotarev invariant of a finite group.
\textit{J. Algebra} {\bf 348} (2011) 302--314.


\bibitem{LiebeckPyber} M.W. Liebeck \& L. Pyber. Upper bounds for the number of conjugacy classes of a finite group. {\em J.~Algebra} {\bf 198} (1997) 538--562.

\bibitem{andrea2}
A. Lucchini. Generating wreath products and their augmentation ideals. \textit{Rend. Sem. Mat. Univ. Padova} {\bf 98} (1997) 67--87.



\bibitem{meneg}
{A. Lucchini \& F. Menegazzo}.
Generators for finite groups with a unique minimal normal subgroup.
{\em  Rend. Sem. Mat. Univ. Padova} {\bf 98} (1997) 173--191.

\bibitem{lms}
A. Lucchini, M. Morigi \& P. Shumyatsky.
Boundedly generated subgroups of finite groups.
\emph{Forum Math.} {\bf 24:4} (2012) 875--887.

\bibitem{MSW}
G. Malle, J. Saxl \& T. Weigel. Generation of classical groups. {\em Geom.~Dedicata} {\bf 49} (194) 85--116.

\bibitem{nina} N.E. Menezes, M. Quick \& C.M. Roney-Dougal. The probability of generating a finite simple group. {\em Israel J. Math.} {\bf 198:1} (2013) 371--392.

\bibitem{roggen} K. W. Roggenkamp, Integral representations and presentations of finite groups,
Lecture notes in Math. 744 (Springer, Berlin 1979)

\bibitem{st} U. Stammbach,
Cohomological characterisations of finite solvable and nilpotent groups,
{\em J. Pure Appl. Algebra} {\bf 11} (1977/78), no. 1--3, 293--301.
\end{thebibliography}
\end{document}